\documentclass[11pt]{article}
\usepackage{amsmath}
\usepackage{amsthm}
\usepackage{latexsym}
\usepackage{tikz}
\usetikzlibrary{automata,positioning}
\usetikzlibrary{chains,fit,shapes}
\usetikzlibrary{calc}
\usetikzlibrary{arrows}
\usepackage{time}             
\usepackage{graphicx, epsfig}
\usepackage[T1]{fontenc}      
\usepackage{amssymb,amsmath}  
\usepackage{palatino}         
\usepackage{multimedia}
\usepackage{subfigure}
\usepackage{mathrsfs}
\usepackage[usenames,dvipsnames]{pstricks}
 \usepackage{epsfig}
 \usepackage{pst-grad} 
\usepackage{pst-plot} 
\usetikzlibrary{fit,matrix}
\tikzset{
mN/.style = {
    draw=#1, semithick, inner sep=0pt}
             }

\newtheorem{theorem}{Theorem}[section]
\newtheorem{definition}[theorem]{Definition}
\newtheorem{lemma}[theorem]{Lemma}
\newtheorem{remark}[theorem]{Remark}
\newtheorem{proposition}[theorem]{Proposition}

\newtheorem{example}[theorem]{Example}

\def\eref#1{(\ref{#1})}

\title{Oriented Riordan graphs and their fractal property\thanks{This work was supported by the National Research Foundation of Korea (NRF) grant funded by the Ministry of Education of Korea (NRF-2019R1I1A1A01044161).}}
\date{}
\author{Ji-Hwan Jung\\{\footnotesize \textit{Center for Educational Research, Seoul National University, Seoul 08826, Republic of Korea}}
\\
{\footnotesize jihwanjung@snu.ac.kr} }

\begin{document}

\maketitle

\begin{abstract}
In this paper, we use the theory of Riordan matrices to introduce
the notion of an oriented Riordan graph. The oriented Riordan graphs
are a far-reaching generalization of the well known and well studied
Toeplitz oriented graphs and tournament. The main focus in this
paper is the study of structural properties of the oriented Riordan
graphs which includes a fundamental decomposition theorem and
fractal property. Finally, we introduce the generalization of the oriented Riordan graph who is called a $p$-Riordan graph.

 \bigskip

\end{abstract}

\section{Introduction}

An {\it oriented graph} is a directed graph having no symmetric pair
of directed edges. Let $G^\sigma$ be a simple graph with an
orientation $\sigma$, which assigns to each edge a direction so that
$G^\sigma$ becomes a directed graph \cite{SW}. With respect to a
labeling, the {\it skew-adjacency matrix} $\mathcal{S}(G^\sigma)$ is
the $(-1,0,1)$-real skew symmetric matrix $[s_{ij}]_{1\le i,j\le n}$
where $s_{ij} = 1$ and $s_{ji}=-1$ if $i\rightarrow j$ is an arc of
$G^\sigma$, otherwise $s_{ij}=s_{ji}= 0.$ A complete oriented graph
is called a {\it tournament}.

An oriented graph $G^\sigma_n$ with $n$ vertices is {\em Riordan} if
there exists a labeling $1,2,\ldots,n$ of $G^\sigma_n$ such that the
lower triangular part of order $n-1$ of the skew-adjacency matrix
$S(G^\sigma_n)$ is of size $n-1$ of some Riordan array
$(g,f)=[\ell_{ij}]_{i,j\ge0}$ over the finite field ${\mathbb Z}_3$
defined as
\begin{eqnarray}\label{e:a
sam}
 \ell_{ij}\equiv[z^i]gf^j\;({\rm mod}\; 3)\;{\rm and}\;\ell_{i,j}\in\{-1,0,1\}
\end{eqnarray}
where $g=\sum_{n\ge0}g_nz^n$ and $f=\sum_{n\ge1}f_nz^n$ are formal
power series over integers ${\mathbb Z}$. By using Riordan language,
the skew-adjacency matrix $\mathcal{S}(G^\sigma_n)$ can be written
as
\begin{eqnarray*}
\mathcal{S}(G^\sigma)\equiv (zg,f)_n-(zg,f)_n^T\;({\rm mod 3})
\end{eqnarray*}
where $2\equiv -1$ ({\rm mod 3). We denote such graph by
$G^\sigma_n(g,f)$, or simply by $G^\sigma_n$ when the Riordan array
$(g,f)$ is understood from the context, or it is not important. For
Riordan graphs, see ~\cite{CJKM1,CJKM2}.

Note that every Riordan array $(g,f)$ over ${\mathbb Z}$ defines the
oriented Riordan graph $G^\sigma(g,f)$ with respect to the labeling
with the same ones as column indices of the the Riordan array
$(g,f)$. For instance, consider the Pascal array given by
$(1/(1-z),z/(1-z))$. The corresponding skew-adjacency matrix  of
order 7 and its oriented Riordan graph
$G_7^\sigma=G^\sigma_7(1/(1-z),z/(1-z))$ are illustrated in Figure
\ref{PG_7}.

\begin{figure}
\begin{center}
\epsfig{file=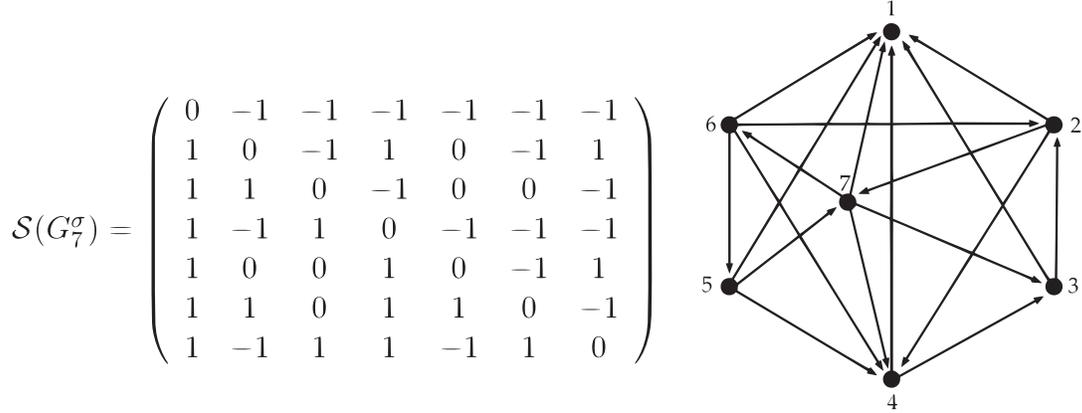,scale=0.18}
\end{center}
\caption{$\mathcal{S}(G^\sigma_7)$ and
$G^\sigma_7=G^\sigma_7(1/(1-z),z/(1-z))$}\label{PG_7}
\end{figure}

Throughout this paper, we write $a\equiv b$ for $a\equiv b\;({\rm
mod\;3})$.

\begin{proposition}\label{num-graphs} The number of oriented Riordan graphs on $n$
vertices is
\begin{align*}
{3^{2(n-1)}+3\over 4}.
\end{align*}
\end{proposition}
\begin{proof} Let $G_n^\sigma=G_n^\sigma(g,f)$ be an oriented Riordan graphs on $n\ge2$
vertices and $i$ be the smallest index such that
$g_i=[z^i]g\not\equiv 0$.

\begin{itemize}
\item If $i\ge n-1$ then $G_n^\sigma$ is the null graph $N_n$.
\item If $0\le i\le n-2$ then we may assume that $g=\sum_{k=i}^{n-2}g_kz^k$ and $f=\sum_{k=1}^{n-2-i}g_kz^k$.
\end{itemize}
Since $g_i\in\{-1,1\}$ and $g_{i+1},\ldots,g_{n-2},f_1,\ldots,
f_{n-i-2}\in\{-1,0,1\}$ it follows that the number of possibilities
to create $n\times n$ skew-adjacency matrix $S(G^\sigma_n)$ is
\begin{align*}
1+2\sum_{i=0}^{n-2}3^{2(n-i-2)}={3^{2(n-1)}+3\over 4}
\end{align*}
where the 1 corresponds to the null graph. \end{proof}

\begin{remark}{\rm It is known \cite{oeis} that the numbers
$1,3,21,183,1641,\ldots,(3^{2(n-1)}+3)/4,\ldots$ count closed walks
of length $2n\;(n\ge1)$ along the edges of a cube based at a vertex.
The numbers are also equal to the numbers of words of length
$2n\;(n\ge1)$ on alphabet $\{0,1,2\}$ with an even number (possibly
zero) of each letter. See the OEIS number A054879.}
\end{remark}

\begin{figure}
\begin{center}
\epsfig{file=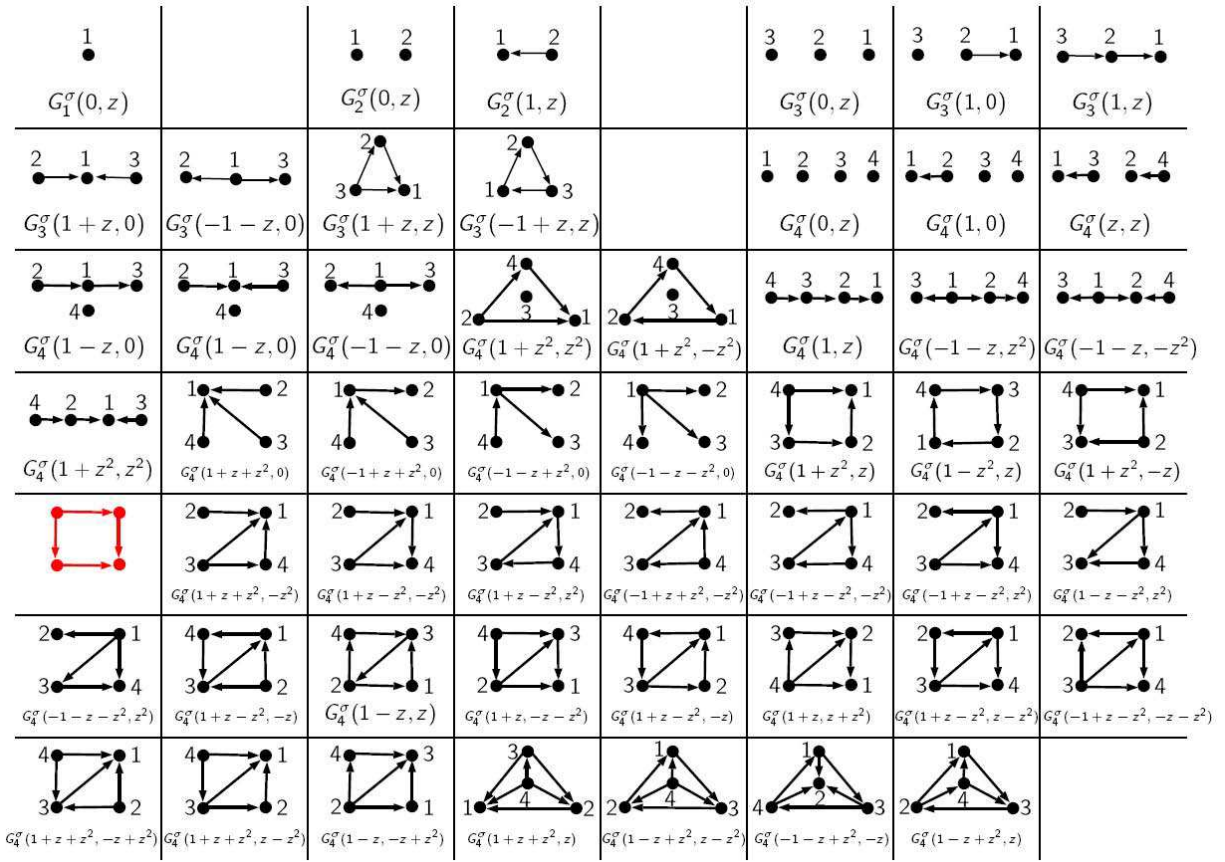,scale=0.5}
\end{center}
\caption{Vertex labeling on nonisomorphic oriented graphs of order
up to 4}\label{Riordan-graphs-up-to-4}
\end{figure}

From Figure \ref{Riordan-graphs-up-to-4}, we can see that the
following graph is the only nonisomorphic oriented graph of order up
to 4 who is not Riordan.
\begin{center}
\epsfig{file=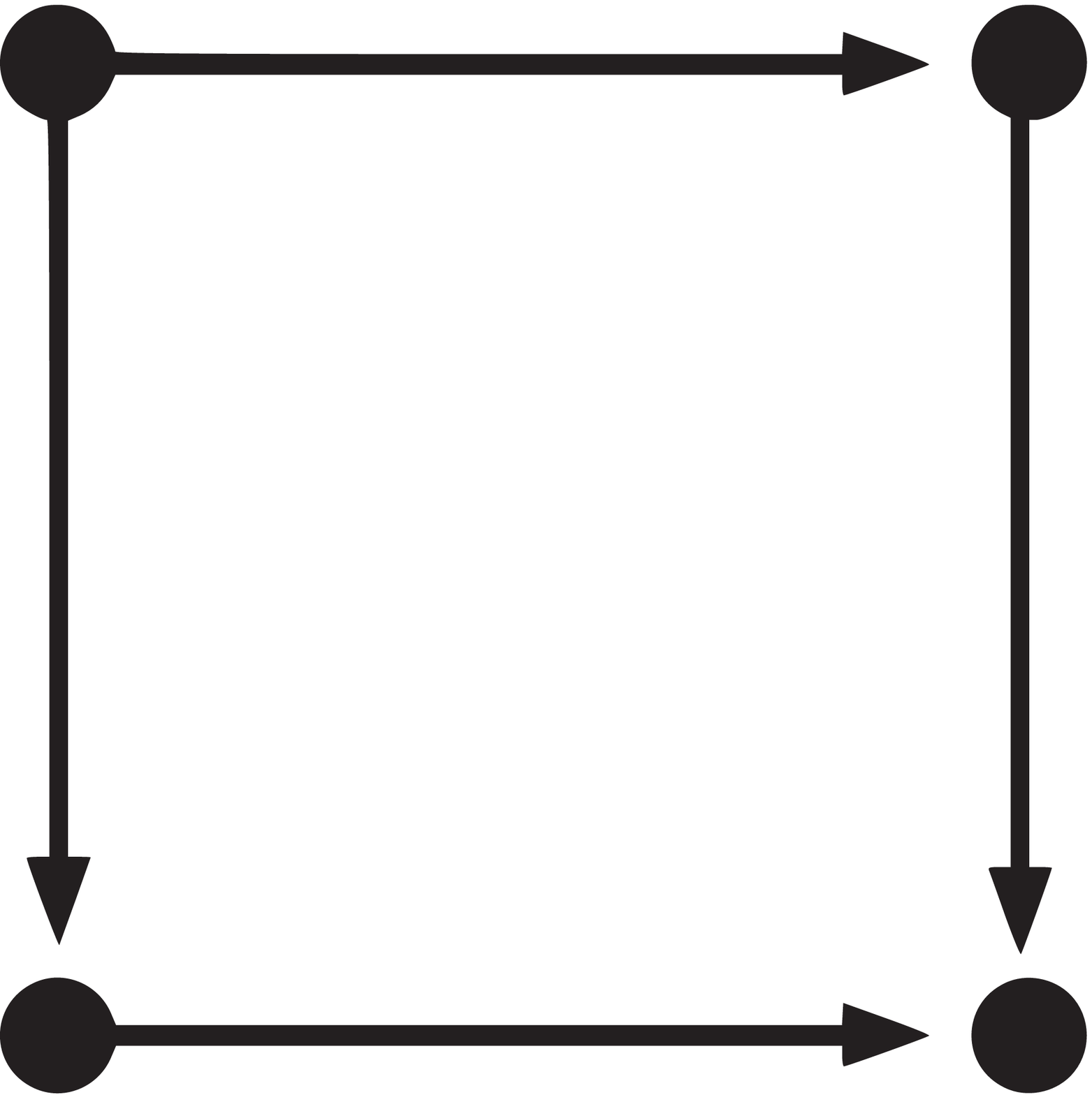,scale=0.15}
\end{center}
Thus the following proposition shows that not all nonisomorphic
oriented graphs on $n$ vertices are Riordan for $n\geq 4$.

\begin{proposition}\label{non-oriented-Riordan-graph} Let $H_n\cong K_{n-1}\cup K_1$ be a graph obtained from a complete graph $K_{n-1}$ by
adding an isolated vertex. Then any orientation of $H_{n}$ for
$n\ge5$ is not an oriented Riordan graph.
\end{proposition}
\begin{proof} Let $H_{n}^\sigma$ be any oriented graph of $H_n$ with an orientation $\sigma$. Suppose that there exist $g$ and $f$ such that a labelled copy
of $H_{n+1}^\sigma$ is the oriented Riordan graph $G_n^\sigma(g,f)$.

Let the isolated vertex be labelled by 1. Since there are no arcs
$1\rightarrow i$ and $i\rightarrow 1\in E(H_{n}^\sigma)$ for
$i=2,\ldots,n$, we have $g=0$ so that $G_n^\sigma(g,f)$ is the null
graph $N_n$. This is a contradiction.

Let $i\neq1$ be the label of the isolated vertex and
$\mathcal{S}(H_{n}^\sigma)=[s_{i,j}]_{1\leq i,j\leq n}$. Then we
obtain
\begin{align*}
(s_{2,1},\ldots,s_{n,n-1})=\left\{
\begin{array}{ll}
(a_1,\ldots,a_{i-2},0,0,a_{i+1},\ldots,a_{n-1})
 & \text{if $i\in\{2,\ldots,n-1\};$} \\
(a_1,\ldots,a_{n-2},0) & \text{if $i=n$}%
\end{array}
\right.
\end{align*}
where $a_{i}\in\{-1,1\}$. This implies that
\begin{itemize}
\item $[z^{i-1}]gf^{i-1}\equiv0$
and $[z^{i}]gf^{i}\not\equiv0$ if $i\in\{2,\ldots,n-2\}$;
\item $[z^{n-4}]gf^{n-4}\not\equiv0$
and $[z^{n-3}]gf^{n-3}\equiv0$ if $i=n-1$;
\item $[z^{n-3}]gf^{n-3}\not\equiv0$
and $[z^{n-2}]gf^{n-2}\equiv0$ if $i=n$.
\end{itemize}
This is also a contradiction. Hence the proof follows.
\end{proof}

\section{Riordan arrays}

 Let $\kappa[[z]]$ be the ring of formal power series in the variable $z$ over an
integral domain $\kappa$.
 If there exists a pair of generating functions $(g,f)\in \kappa[[z]]\times \kappa[[z]]$, $f(0)=0$ such
 that for $j\ge 0$,
\begin{eqnarray*}
 gf^j=\sum_{i\ge0}\ell_{i,j}z^i,
\end{eqnarray*}
then the matrix  $L=[\ell_{ij}]_{i,j\ge0}$ is called a  {\it Riordan
matrix} (or, a {\it Riordan array})  over $\kappa$ generated by $g$
and $f$. Usually, we write $L=(g,f)$. Since $f(0)=0$, every Riordan
matrix $(g,f)$ is infinite and a lower triangular matrix. If a
Riordan matrix is invertible, it is called {\it proper}. Note that
$(g,f)$ is invertible if and only if $g(0)\ne0$, $f(0)=0$ and
$f^\prime(0)\ne0$.

If we multiply $(g,f)$ by a column vector $(c_0,c_1,\ldots)^T$ with
the generating function $\Phi$ over an integral domain $\kappa$ with
characteristic zero, then the resulting column vector has the
generating function $g\Phi(f)$.
This property is known as the {\it fundamental theorem of Riordan
matrices} ({\it FTRM}). Simply, we write the FTRM as
$(g,f)\Phi=g\Phi(f)$. This leads to the multiplication of Riordan
matrices, which can be described in terms of generating functions as
\begin{eqnarray}\label{e:Rm}
(g,f)*(h,\ell)=(gh(f),\ell(f)).
\end{eqnarray}
The set of all proper Riordan matrices under the above {\em Riordan
multiplication} forms a group called the \textit{Riordan group}. The
identity of the group is $(1,z)$, the usual identity matrix and
$(g,f)^{-1}=({1/g(\overline{f})},\overline{f})$ where $\overline{f}$
is the {\em compositional inverse} of $f$, i.e.\
$\overline{f}(f(z))=f(\overline{f}(z))=z$.

The {\em leading principal matrix of order $n$} of $(g,f)$ is
denoted by $(g, f)_n$. If $\kappa={\mathbb Z}$ then the fundamental
theorem gives
\begin{align}\label{e:ftbrm}
(g,f)\Phi\equiv g\Phi(f).
\end{align}

It is known \cite{MRSV} that an infinite lower triangular matrix
$L=[\ell_{i,j}]_{i,j\ge0}$ is a Riordan matrix $(g,f)$ with
$[z^1]f\neq0$ if and only if there is a unique sequence
$(a_0,a_1,\ldots)$ with $a_0\neq0$ such that, for $i\ge j\ge0$,
\begin{align*}
\ell_{i,0}&=[z^i]g;\\
\ell_{i+1,j+1}&=a_0\ell_{i,j}+a_1\ell_{i,j+1}+\cdots+a_{i-j}\ell_{i,i}.
\end{align*}
This sequence is called  the {\it $A$-sequence} of the Riordan
array. Also, if $L=(g,f)$ then
\begin{align}\label{e:eq12}
f=zA(f),\quad{\rm or\ equivalently}\quad A=z/\bar f
\end{align}
where $A$ is the generating function of the $A$-sequence of $(g,f)$.
In particular, if $L$ is a Riordan array $(g,f)$ over ${\mathbb
Z}_3$ with $f'(0)=1$ then the sequence is called the {\it ternary
$A$-sequence} $(1,a_1,a_2,\ldots)$ where $a_k\in\{-1,0,1\}$.

\section{Structural properties of oriented Riordan graphs}

A {\em fractal} is an object exhibiting similar patterns at
increasingly small scales. Thus, fractals use the idea of a detailed
pattern that repeats itself.

%
In this section, we show that every oriented Riordan graph
$G_n^\sigma(g,f)$ with $f'(0)=1$ has fractal properties by using the
notion of the $A$-sequence of a Riordan matrix. The set of labelings
$1,2,\ldots,n$ of the graph $G_n^\sigma(g,f)$ is`denoted as $[n]$.

\begin{definition}{\rm Let $G^\sigma$ be an oriented Riordan graph. A pair of vertices $\{k,t\}$ in $G^\sigma$ is
a {\em cognate pair} with a pair of vertices $\{i,j\}$ in $G^\sigma$
if
\begin{itemize}
\item $|i-j|=|k-t|$ and
\item $i\rightarrow j$ is an arc of $G^\sigma$ if and only if $k\rightarrow t$ is an arc of $G^\sigma$.
\end{itemize}
The set of all cognate pairs of $\{i,j\}$ is denoted by cog$(i,j)$.}
\end{definition}

\begin{lemma}\label{p-power}
Let $g,f\in\mathbb{Z}[[z]]$ with $f(0)=0$. For a prime $p$, we
obtain
\begin{eqnarray*}
g^{p^k}(f)\equiv g(f^{p^k})\;({\rm mod}\;p) \;\textrm{for any
integer $k\ge0$.}
\end{eqnarray*}
\end{lemma}

Throughout this paper, we write $a\equiv b$ for $a\equiv b\;({\rm
mod\;3})$.

The following theorem gives a relationship between cognate pairs and
the $A$-sequence of a Riordan array.

\begin{theorem}\label{e:cognate} Let $G^\sigma_n(g,f)$ be an oriented Riordan graph of order $n$ where $f\ne z$ and $[z^1]f=1$.
If the ternary $A$-sequence of $(g,f)$ is of the form
\begin{eqnarray}\label{e:ta}
A=(a_k)_{k\ge0}=(1,\underbrace{0,\ldots,0}_{\ell\ge0 \;{\rm
times}},a_{\ell+1},a_{\ell+2},\ldots),\;\;a_0,a_{\ell+1}\not\equiv0
\end{eqnarray} then
\begin{align*}
\mbox{\em cog}(i,j)=\left\{\left\{i+m3^{s}, j+m3^{s}\right\}\ |\
i+m3^{s},\ j+m3^{s}\in[n]\right\}
\end{align*}
where $s\ge0$ is an integer such that
$3\lfloor(|i-j|-1)/3^s\rfloor\le \ell$.
\end{theorem}
\begin{proof} Let $\mathcal{S}(G_n^\sigma)=(r_{i,j})_{1\leq i,j\leq n}$
be the skew-adjacency matrix of $G_n^\sigma(g,f)$. Without loss of
generality, we may assume that $i>j\ge1$. By Lemma \ref{p-power}, we
obtain \begin{align} r_{i+3^s,j+3^s}&\equiv\left[z^{i+3^s-2}\right]
gf^{j+3^s-1}=\left[z^{i+3^s-2}\right]
gf^{j-1}\left(zA(f)\right)^{3^s}\nonumber\\
&\equiv\left[z^{i-2}\right]gf^{j-1}A(f^{3^{s}})=\left[z^{i-2}\right]\sum_{k\ge0}a_kgf^{j-1+k3^s}\nonumber\\
&\equiv \sum_{k=0}^{\alpha}a_kr_{i,j+k3^s}\label{e:eq6}.
\end{align}
where $\alpha={\rm max}\{k\in\mathbb{N}_0\;|\;0\le k3^s\le
i-j-1\}=\lfloor(i-j-1)/3^s\rfloor$. Since $a_0=1$ and $a_k=0$ for
$1\le k\le \ell$, it follows that $\alpha\le \ell$ if and only if
\begin{align}\label{e:eq7}
r_{i+3^s,j+3^s}\equiv r_{i,j}.
\end{align}
Now, let $s\ge0$ be an integer with
$\lfloor(i-j-1)/3^s\rfloor\le\ell$. By \eref{e:eq7}, $i$ is
adjacency to $j$ with $i\rightarrow j$ in $G_n^\sigma$ if and only
if $i+3^s$ is adjacency to $j+3^s$ with $i+3^s\rightarrow j+3^s$ in
$G_n^\sigma$. It implies that $i+3^s$ is adjacency to $j+3^s$ with
$i\rightarrow j$ in $G_n^\sigma$ if and only if $i+2\cdot3^s$ is
adjacency to $j+2\cdot3^s$ with $i+2\cdot3^s\rightarrow j+2\cdot3^s$
in $G_n^\sigma$. By repeating this process, we obtain the desired
result.
\end{proof}

The following theorem shows that if $(g,f)$ is a Riordan array over
integers where $f\ne z$ and $f'(0)=1$ then every oriented Riordan
graph $G_n^\sigma(g,f)$ has a fractal property.

\begin{theorem}\label{ORG-fractal}
Let $G^\sigma_n(g,f)$ be the same oriented Riordan graph in Theorem
\ref{e:cognate}. If $a_0=1$ in \eref{e:ta} then $G_n^\sigma$ has the
following fractal properties for each $s\ge 0$ and
$k\in\{0,\ldots,\ell\}$:
\begin{itemize}
\item[{\rm(i)}] $\left<\{1,\ldots,(k+1)3^{s}+1\}\right>\cong
\left<\{\alpha(k+1)3^{s}+1,\ldots,(\alpha+1)(k+1)3^{s}+1\}\right>$
\item[{\rm(ii)}]$\left<\{1,\ldots,(k+1)3^{s}\}\right>\cong
\left<\{\alpha(k+1)3^{s}+1,\ldots,(\alpha+1)(k+1)3^{s}\}\right>$
\end{itemize}
where $\alpha\ge1$.
\end{theorem}
\begin{proof} Let $i,j\in\{1,\ldots,(k+1)3^{s}+1\}\in V(G_n^\sigma)$ with $0\le
k\le\ell$. Since
\begin{align*}
\left\lfloor{|i-j|-1\over
3^s}\right\rfloor\le\left\lfloor{(k+1)3^s-1\over
3^s}\right\rfloor\le\ell,
\end{align*}
it follows from Theorem \ref{e:cognate} that
\begin{align}\label{e:con2}
\left\{i+\alpha(k+1)3^{s},
j+\alpha(k+1)3^{s}\right\}\in\textrm{cog}(i,j).
\end{align}
Thus $i$ is adjacent to $j$ with $i\rightarrow j$ in $G_n^\sigma$ if
and only if $i+\alpha(k+1)3^{s}$ is adjacent to $j+\alpha(k+1)3^{s}$
with $i+\alpha(k+1)3^{s}\rightarrow j+\alpha(k+1)3^{s}$ in
$G_n^\sigma$. Hence we obtain (i). Similarly we obtain (ii). Hence
the proof follows.\end{proof}

Note that real skew symmetric matrices with respect to different
labelings are permutationally similar.

\begin{example} {\rm Let us consider an oriented Pascal graph
$PG_n^\sigma=G_n^\sigma(1/(1-z),z/(1-z))$. Since by \eref{e:eq12}
its $A$-sequence is given by $(1,1,0,\ldots)$, i.e.\ $\ell=0$ so
that $k=0$, it follows from Theorem \ref{ORG-fractal} that
\begin{align}
\left<\{1,\ldots,3^{s}+1\}\right>\cong
\left<\{\alpha3^{s}+1,\ldots,(\alpha+1)3^{s}+1\}\right>
\end{align}
For instance, when $n=19$, $s=2$ and $\alpha=1$
$\left<\{1,\ldots,10\}\right>\cong \left<\{10,\ldots,19\}\right>$
i.e., if $\mathcal{S}(PG_{19}^\sigma)=(p_{n,k})_{1\le n\le19}$ then
$(p_{n,k})_{1\le n,k\le10}=(p_{n,k})_{10\le n,k\le19}$, see Figure
\ref{PG19}.}
\end{example}

\begin{figure}
\begin{center}
\epsfig{file=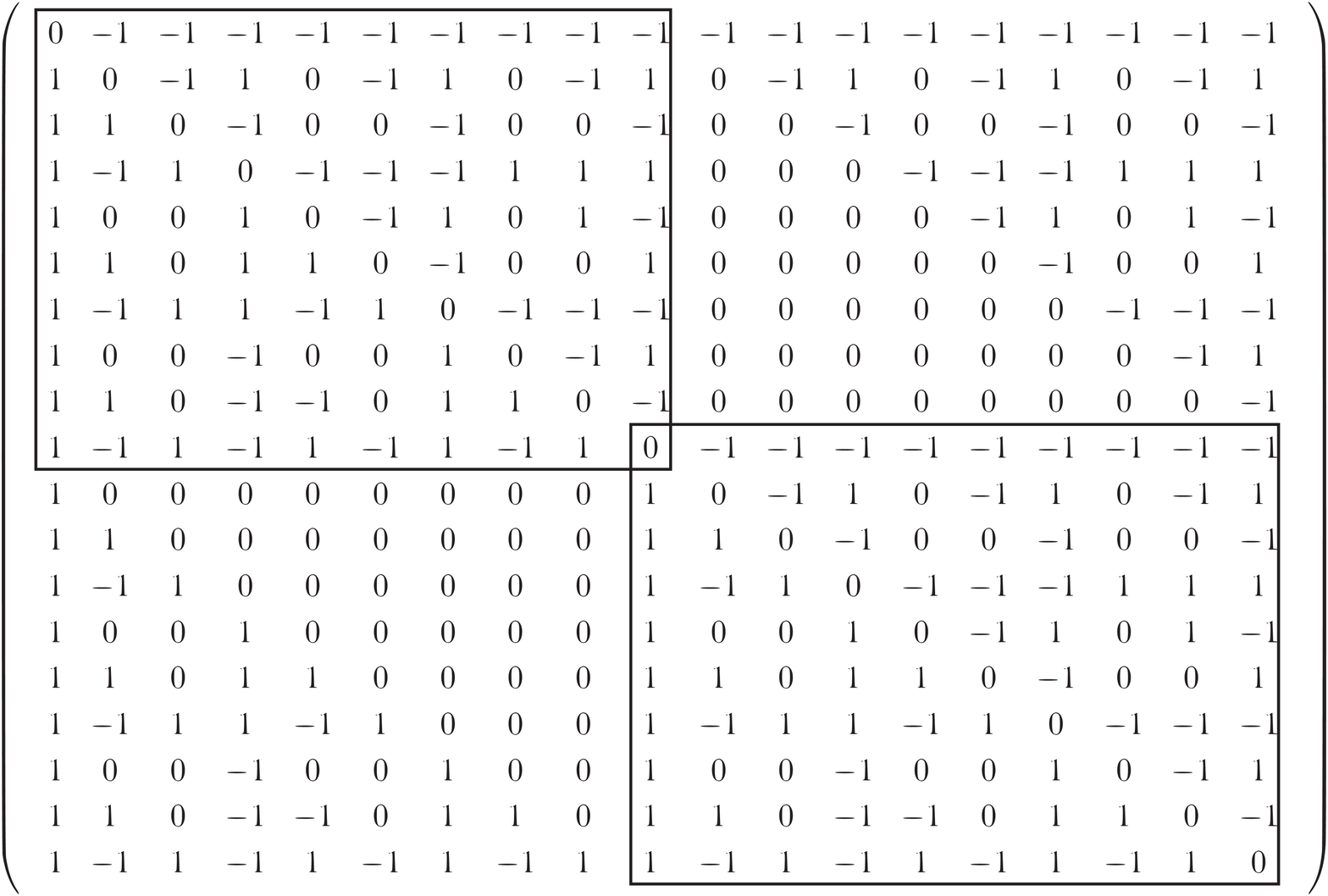,scale=0.22}
\end{center}
\caption{$\mathcal{S}(PG_{19}^\sigma)$}\label{PG19}
\end{figure}

\begin{lemma}\label{p-1-derivitive}
Let $h=\sum_{n\ge0}h_nz^n\in\mathbb{Z}[[z]]$. For a prime $p$, we
obtain
\begin{eqnarray*}
{d^{p-1}\over
dz^{p-1}}h(\sqrt[p]{z})\equiv-\sum_{k\ge0}h_{(k+1)p-1}z^{k}\;({\rm
mod}\;p).
\end{eqnarray*}
\end{lemma}
\begin{proof}
By applying $(p-1)$th derivative of $h(z)$, we obtain
\begin{align*}
[z^m]{d^{p-1}\over
dz^{p-1}}h(z)&=m(m-1)\cdots(m-p+2)h_{m+p-1}=(p-1)!{m\choose
p-1}h_{m+p-1}.
\end{align*}
By the Wilson theorem and the Lucas theorem, the right had side of
the above equation can be written as
\begin{align*}
(p-1)!{m\choose p-1}h_{m-p+1}\equiv_p-{m\choose
p-1}h_{m+p-1}\equiv_p\left\{
\begin{array}{ll}
-h_{m+p-1}
 & \text{if $m=(k+1)p-1$} \\
0 & \text{otherwise}%
\end{array}
\right.
\end{align*}
where $k\ge0$ and $a\equiv_p b$ denotes $a\equiv b\;({\rm mod}\;p)$.
This implies
\begin{align*}
{d^{p-1}\over dz^{p-1}}h(z)\equiv_p-\sum_{k\ge0}h_{(k+1)p-1}z^{pk}.
\end{align*}
Hence the proof follows.
\end{proof}

We now consider the vertex set $V=[n]$ of an oriented Riordan graph
$G_n^\sigma$ of order $n\ge 3$. Then $V$ can be partitioned into
three subsets $V_1$, $V_2$ and $V_3$ where
$V_i=\{j\in[n]\;|\;i\equiv j\}$ for $j=1,2,3$. There exists a
permutation matrix $P$ such that
\begin{align*}
\mathcal{S}(G^\sigma)=P^{T}\left(
\begin{array}{ccc}
\mathcal{S}(\left<V_1\right>) & B_{1,2} & B_{1,3}  \\
-B_{1,2}^T & \mathcal{S}(\left<V_2\right>) & B_{2,3} \\
-B_{1,3}^T  & -B_{2,3}^T & \mathcal{S}(\left<V_3\right>)
\end{array}
\right)P
\end{align*}
where $V_i$ and $V_j$, $i\neq j$, are mutually nonempty disjoint
subsets of $V$ such that $V_1\cup V_2\cup V_3=V$.

\begin{theorem}[Oriented Riordan Graph Decomposition]\label{e:th}  Let $G_n^\sigma=G_n^\sigma(g,f)$ be an oriented Riordan graph of order $n\ge 3$ with the vertex set
$V=V_1\cup V_2\cup V_3$ where $V_i=\{j\in[n]\;|\;j\equiv i\}$. Then
its skew-adjacency matrix $\mathcal{S}(G_n^{\sigma})$ is
permutationally similar to the $3\times3$ block matrix:
\begin{eqnarray}\label{e:bm}
\left(
\begin{array}{cccc}
X_1 & B_{1,2} & B_{1,3}  \\
-B_{1,2}^T & X_2 & B_{2,3} \\
-B_{1,3}^T & -B_{2,3}^T & X_3
\end{array}
\right)
\end{eqnarray}
where $X_i=\mathcal{S}(\left<V_i\right>)$, the skew-adjacency matrix
of the induced subgraph of $G_n^\sigma$ by $V_i$, $i=1,2,3$.

In particular, $\left<V_i\right>$ is isomorphic to the oriented
Riordan graph of order $\ell_i=\lfloor (n-i)/3\rfloor+1$ given~by
\begin{align*}
G^\sigma_{\ell_i}\left(-{d^{2}\over dz^{2}}\left({gf^{i-1}\over
z^{i-1}}\right)(\sqrt[3]{z}),f(z)\right),
\end{align*}
and $B_{i,j}$ representing the edges between $V_i$ and $V_j$ can be
expressed as the sum of two Riordan matrices as follows:
\begin{align*}
B_{i,j}\equiv\left(-z{d^{2}\over dz^{2}}\left({gf^{j-1}\over
z^{i-1}}\right)(\sqrt[3]{z}),f(z)\right)_{\ell_i
\times\ell_j}+\left({d^{2}\over
dz^{2}}(z^{4-j}gf^{i-1})(\sqrt[3]{z}),f(z)\right)_{\ell_j\times\ell_i}^{T}.
\end{align*}
\end{theorem}
\begin{proof} Using the $n\times n$ permutation matrix defined as
$$
P=\left[e_{1}|e_{4}|\cdots|e_{3\lceil n/3\rceil
-2}|\cdots|e_{3}|e_{6} \;|\; \cdots \;|\; e_{3\lfloor n/3\rfloor
}\right] ^{T}
$$
where $e_{i}$ is the elementary column vector with the $i$th entry
being $1$ and the others entries being $0$, it can be shown that
$P\mathcal{S}(G_n^\sigma)P^{T}$ is equal to the block matrix in
\eref{e:bm}.

Taking into account the form of $P$, clearly $X_i$ is the
skew-adjacency matrix of the induced subgraph $\left<V_i\right>$ of
order $\ell_i=\lfloor (n-i)/3\rfloor+1$ in
$G_n^\sigma=G_n^\sigma(g,f)$. Let
$\mathcal{S}\left(G_n^\sigma\right)=[r_{i,j}]_{1\leq i,j\leq n}$.
Since $f=zA(f)$ where $A(z)\in \mathbb{Z}[[z]] $ is the generating
function of the $A$-sequence $(a_0,a_1,\ldots)$ for the Riordan
matrix $(g,f)$, it follows from Lemma \ref{p-power} that for
$i>j\geq p+1$
\begin{align}\label{e:eq}
r_{i,j} &\equiv\left[ z^{i-2}\right] gf^{j-1}=\left[ z^{i-2}\right]
gf^{j-4}\left(zA(f)\right)^3 \equiv \left[ z^{i-5}\right]
gf^{j-4}A\left( f^{3}\right)\nonumber \\
&=\left[ z^{i-5}\right] \left(
a_{0}gf^{j-4}+a_{1}gf^{j-1}+a_{2}gf^{j+2}+\cdots \right) \nonumber \\
&=\sum_{k=0}^{\lfloor(i-j-1)/3\rfloor}a_kr_{i-3,j-3+3k}.
\end{align}
Since $X_i=[r_{3(u-1)+i,3(v-1)+i}]_{1\le u,v\le\ell_i}$ and by Lemma
\ref{p-1-derivitive} $${d^{2}\over dz^{2}}\left({gf^{i-1}\over
z^{i-1}}\right)={d^{2}\over
dz^{2}}\left(\sum_{k\ge0}r_{k+i+1,i}z^k\right)\equiv-\sum_{k\ge0}r_{3(k+1)+i,i}z^{3k},$$
it follows from \eref{e:eq} that
\begin{align*}
\left<V_i\right>\cong G^\sigma_{\ell_i}\left(-{d^{2}\over
dz^{2}}\left({gf^{i-1}\over
z^{i-1}}\right)(\sqrt[3]{z}),f(z)\right).
\end{align*}

Let
\begin{align*}
L_{i,j}=\left[
\begin{array}{cccc}
0 & 0 & 0 & \cdots \\
r_{3+i,j} & 0 & 0 & \cdots  \\
r_{6+i,j} & r_{6+i,3+j} & 0 & \cdots  \\
r_{9+i,j} & r_{9+i,3+j} & r_{9+i,6+j} & \ddots  \\
\vdots  & \vdots  & \vdots   & \ddots
\end{array}
\right]\;{\rm and}\; U_{i,j}=\left[
\begin{array}{ccccc}
r_{j,i} & 0 & 0 & 0 &\cdots \\
r_{3+j,i} & r_{3+j,3+i} & 0 & 0 & \cdots \\
r_{6+j,i} & r_{6+j,3+i} & r_{6+j,6+i} & 0 & \cdots\\
r_{9+j,i} & r_{9+j,3+i} & r_{9+j,6+i} & r_{9+j,9+i} & \ddots\\
\vdots  & \vdots  & \vdots  & \vdots  & \ddots
\end{array}
\right].
\end{align*}
Since by Lemma \ref{p-1-derivitive} we obtain
\begin{align*}
{d^{2}\over dz^{2}}\left({gf^{j-1}\over z^{i-1}}\right)={d^{2}\over
dz^{2}}\left(\sum_{k\ge0}r_{k+i+1,j}z^k\right)\equiv-\sum_{k\ge0}r_{3(k+1)+i,j}z^{3k}
\end{align*}
and
\begin{align*}
{d^{2}\over dz^{2}}(z^{4-j}gf^{i-1})={d^{2}\over
dz^{2}}\left(\sum_{k\ge0}r_{k+j-2,i}z^k\right)\equiv-\sum_{k\ge0}r_{3k+j,i}z^{3k},
\end{align*}
we obtain
\begin{align*}
L_{i,j}=\left(-z{d^{2}\over dz^{2}}\left({gf^{j-1}\over
z^{i-1}}\right)(\sqrt[3]{z}),f(z)\right)\;{\rm
and}\;U_{i,j}=\left(-{d^{2}\over
dz^{2}}(z^{4-j}gf^{i-1})(\sqrt[3]{z}),f(z)\right).
\end{align*}
One can see that $B_{i,j}=(L_{i,j}-U_{i,j}^T)_{\ell_i\times\ell_j}$
where $(L_{i,j}-U_{i,j}^T)_{\ell_i\times\ell_j}$ is the
$\ell_i\times\ell_j$ leading principal matrix of $L_{i,j}-U_{i,j}$.
Hence the proof follows.
\end{proof}

\begin{example} {\rm Let us consider an oriented Pascal graph
$PG_{13}^\sigma=G_{13}^\sigma(1/(1-z),z/(1-z))$. Then we have
{\small\begin{align}\label{PG13}
\mathcal{S}(PG_{13}^\sigma)=\left(\begin{array}{ccccccccccccc}
0 & -1 & -1 & -1 & -1 & -1 & -1 & -1 & -1 & -1 & -1 & -1 & -1 \\
1 & 0 & -1 & 1 & 0 & -1 & 1 & 0 & -1 & 1 & 0 & -1 & 1 \\
1 & 1 & 0 & -1 & 0 & 0 & -1 & 0 & 0 & -1 & 0 & 0 & -1 \\
1 & -1 & 1 & 0 & -1 & -1 & -1 & 1 & 1 & 1 & 0 & 0 & 0 \\
1 & 0 & 0 & 1 & 0 & -1 & 1 & 0 & 1 & -1 & 0 & 0 & 0 \\
1 & 1 & 0 & 1 & 1 & 0 & -1 & 0 & 0 & 1 & 0 & 0 & 0 \\
1 & -1 & 1 & 1 & -1 & 1 & 0 & -1 & -1 & -1 & 0 & 0 & 0 \\
1 & 0 & 0 & -1 & 0 & 0 & 1 & 0 & -1 & 1 & 0 & 0 & 0 \\
1 & 1 & 0 & -1 & -1 & 0 & 1 & 1 & 0 & -1 & 0 & 0 & 0 \\
1 & -1 & 1 & -1 & 1 & -1 & 1 & -1 & 1 & 0 & -1 & -1 & -1 \\
1 & 0 & 0 & 0 & 0 & 0 & 0 & 0 & 0 & 1 & 0 & -1 & 1 \\
1 & 1 & 0 & 0 & 0 & 0 & 0 & 0 & 0 & 1 & 1 & 0 & -1 \\
1 & -1 & 1 & 0 & 0 & 0 & 0 & 0 & 0 & 1 & -1 & 1 & 0
\end{array}\right).
\end{align}}For a permutation matrix
$P=\left[e_{1}|e_{4}|e_{7}|e_{10}|e_{13}|e_{2}|e_{5}|e_{8}|e_{11}|e_{3}|e_{6}|e_{9}|e_{12}\right]
^{T}$, we obtain {\small\begin{align}\label{PPG13PT}
P\mathcal{S}(PG_{13}^\sigma)P^T=\left(\begin{array}{ccccc|cccc|cccc}
0 & -1 & -1 & -1 & -1 & -1 & -1 & -1 & -1 & -1 & -1 & -1 & -1 \\
1 & 0 & -1 & 1 & 0 & -1 & -1 & 1 & 0 & 1 & -1 & 1 & 0 \\
1 & 1 & 0 & -1 & 0 & -1 & -1 & -1 & 0 & 1 & 1 & -1 & 0 \\
1 & -1 & 1 & 0 & -1 & -1 & 1 & -1 & -1 & 1 & -1 & 1 & -1 \\
1 & 0 & 0 & 1 & 0 & -1 & 0 & 0 & -1 & 1 & 0 & 0 & 1 \\\hline
1 & 1 & 1 & 1 & 1 & 0 & 0 & 0 & 0 & -1 & -1 & -1 & -1 \\
1 & 1 & 1 & -1 & 0 & 0 & 0 & 0 & 0 & 0 & -1 & 1 & 0 \\
1 & -1 & 1 & 1 & 0 & 0 & 0 & 0 & 0 & 0 & 0 & -1 & 0 \\
1 & 0 & 0 & 1 & 1 & 0 & 0 & 0 & 0 & 0 & 0 & 0 & -1 \\\hline
1 & -1 & -1 & -1 & -1 & 1 & 0 & 0 & 0 & 0 & 0 & 0 & 0 \\
1 & 1 & -1 & 1 & 0 & 1 & 1 & 0 & 0 & 0 & 0 & 0 & 0 \\
1 & -1 & 1 & -1 & 0 & 1 & -1 & 1 & 0 & 0 & 0 & 0 & 0 \\
1 & 0 & 0 & 1 & -1 & 1 & 0 & 0 & 1 & 0 & 0 & 0 & 0
\end{array}\right).
\end{align}}Since ${1\over (1-z)^3}\equiv {1\over 1-z^3}$ and
\begin{align*}
{d^{2}\over dz^{2}}\left({gf^{i-1}\over z^{i-1}}\right)={d^{2}\over
dz^{2}}\left({1\over (1-z)^i}\right)={i(i+1)\over (1-z)^{i+2}},
\end{align*}
it follows from Theorem \ref{e:th} that we may check
\begin{align*}
\left<V_1\right>\cong G_{5}^\sigma\left({1\over 1-z},{z\over
1-z}\right)= PG_5^\sigma\;{\rm
and}\;\left<V_2\right>\cong\left<V_3\right>\cong
G_{4}^\sigma\left(0,{z\over 1-z}\right)\cong N_4
\end{align*}
where $V_1=\{1,4,7,10,13\}$, $V_2=\{2,5,8,11\}$ and
$V_3=\{3,6,9,12\}$. Since
\begin{align*}
{d^{2}\over dz^{2}}\left({gf^{2-1}\over z^{1-1}}\right)&=
{d^{2}\over dz^{2}}\left({z\over
(1-z)^2}\right)={2(2+z)\over (1-z)^4}\equiv {1\over (1-z)^3}\equiv {1\over 1-z^3},\\
{d^{2}\over dz^{2}}(z^{4-2}gf^{1-1})&={d^{2}\over dz^{2}}{z^2\over
1-z}={2\over (1-z)^3}\equiv-{1\over 1-z^3},
\end{align*}
it follows from Theorem \ref{e:th} that we may check also
\begin{align*}
B_{1,2}&\equiv\left(-{z\over 1-z},{z\over
1-z}\right)_{5,4}+\left(-{1\over 1-z},{z\over 1-z}\right)_{4,5}^T\\
&=\left(\begin{array}{cccc}
0 & 0 & 0 & 0 \\
-1 & 0 & 0 & 0 \\
-1 & -1 & 0 & 0 \\
-1 & -2 & -1 & 0 \\
-1 & -3 & -3 & -1
\end{array}\right)
+ \left(\begin{array}{ccccc}
-1 & 0 & 0 & 0 & 0 \\
-1 & -1 & 0 & 0 & 0 \\
-1 & -2 & -1 & 0 & 0 \\
-1 & -3 & -3 & -1 & 0%
\end{array}\right)^T\\
&\equiv\left(
\begin{array}{cccc}
-1 & -1 & -1 & -1 \\
-1 & -1 & 1 & 0 \\
-1 & -1 & -1 & 0 \\
-1 & 1 & -1 & -1 \\
-1 & 0 & 0 & -1
\end{array}
\allowbreak \right).
\end{align*}
Similarly, we may check that Theorem \ref{e:th} holds for $B_{1,3}$
and $B_{2,3}$.}
\end{example}

\begin{theorem}\label{e:th13}
Let $G_n^\sigma=G_n^\sigma(g,f)$ be an oriented Riordan graph of
order $n\ge 3$ and $V_i=\{j\in[n]\;|\;j\equiv i\}$, $i=1,2,3$. Then
\begin{itemize}
\item[{\rm(i)}] For $n=3k$ ($k\ge1$), $\left\langle V_i\right\rangle \cong \left\langle
V_j\right\rangle$ if and only if $[z^{3m+i-2}]gf^{i-1}\equiv
[z^{3m+j-2}]gf^{j-1}$ for all $m\ge1$.
\item[{\rm(ii)}] The induced subgraph $\left\langle V_{i}\right\rangle$ is
a null graph if and only if $[z^{3m+i-2}]gf^{i-1}\equiv 0$ for all
$m\ge1$.
\item[{\rm(iii)}] $G_{n}^\sigma$ is a 3-partite graph with parts
$V_1,V_2,V_{3}$ if and only if $[z^{3m+i-2}]gf^{i-1}\equiv0$ for all
$m\ge1$ and $i=1,2,3$.
\item[{\rm(iv)}] For $j>i$, there are no arcs between a vertex $u\in V_i$ and a vertex $v\in V_j$ if and only if
$[z^{3m+i-2}]gf^{j-1}\equiv[z^{3(m-1)+j-2}]gf^{i-1}\equiv 0$ for all
$m\ge1$.
\end{itemize}
\end{theorem}
\begin{proof} (i) Let $n=3k$ with $k\ge1$. From Theorem~\ref{e:th},
$\left\langle V_i\right\rangle \cong \left\langle V_j\right\rangle$
if and only if the matrices $X_i$ and $X_j$ in the block matrix in
\eref{e:bm} are given by
\begin{align*}
{d^{2}\over dz^{2}}\left({gf^{i-1}\over z^{i-1}}\right)\equiv
{d^{2}\over dz^{2}}\left({gf^{j-1}\over
z^{j-1}}\right)&\;\Leftrightarrow\;[z^{3m-1}]{gf^{i-1}\over
z^{i-1}}\equiv[z^{3m-1}]{gf^{j-1}\over z^{j-1}}\\
&\;\Leftrightarrow\;[z^{3m+i-2}]{gf^{i-1}}\equiv[z^{3m+j-2}]{gf^{j-1}}
\end{align*}
for all $m\ge1$ which proves (i).

(ii) From Theorem \ref{e:th}, the induced subgraph $\left\langle
V_i\right\rangle$ is a null graph if and only if the matrix $X_i$ in
\eref{e:bm} is a zero matrix, i.e.\
\begin{align}\label{ii-proof}
{d^{2}\over dz^{2}}\left({gf^{i-1}\over z^{i-1}}\right)\equiv
0\;\Leftrightarrow\;[z^{3m-1}]{gf^{i-1}\over
z^{i-1}}\equiv0\;\Leftrightarrow\;[z^{3m+i-2}]{gf^{i-1}}\equiv0
\end{align}
for all $m\ge1$ which proves (ii).

(iii) From Theorem \ref{e:th}, $G_{n}^\sigma$ is a 3-partite graph
with parts $V_1,V_2,V_{3}$ if and only if the matrices $X_1$, $X_2$
and $X_3$ in \eref{e:bm} are zero matrices.  Hence by
\eref{ii-proof} we obtain desired result.

(iv) Form Theorem \ref{e:th}, there are no arcs between a vertex
$u\in V_i$ and a vertex $v\in V_j$ with $j>i$ if and only if the
matrix $B_{i,j}$ in \eref{e:bm} is a zero matrices, i.e.\
\begin{align*}
{d^{2}\over dz^{2}}\left({gf^{j-1}\over
z^{i-1}}\right)\equiv0\;\;{\rm and}\;\;{d^{2}\over
dz^{2}}(z^{4-j}gf^{i-1})\equiv0\;&\Leftrightarrow\;
[z^{3m-1}]{gf^{j-1}\over
z^{i-1}}\equiv0\;\textrm{and}\;[z^{3m-1}]z^{4-j}gf^{i-1}\equiv0\\
&\Leftrightarrow[z^{3m+i-2}]gf^{j-1}\equiv[z^{3(m-1)+j-2}]gf^{i-1}\equiv
0
\end{align*}
for all $m\ge1$. Hence the proof follows.
\end{proof}
\begin{example}
{\rm Let us consider an oriented Pascal graph
$PG_n^\sigma=G_n^\sigma(1/(1-z),z/(1-z))$. Since for any $m\ge1$
\begin{align*}
[z^{3m+2-2}]{gf^{2-1}}&=[z^{3m}]{z\over
(1-z)^2}=[z^{3m}]\sum_{k\ge1}kz^k=3m\equiv0;\\
[z^{3m+3-2}]{gf^{3-1}}&=[z^{3m}]{z^2\over
(1-z)^3}=[z^{3m}]\sum_{k\ge2}{k\choose
2}z^k={3m(3m-1)\over2}\equiv0,
\end{align*}
it follows from (ii) of Theorem \ref{e:th13} that induced subgraphs
$\left<V_2\right>$ and $\left<V_3\right>$ are null. For instance,
when $n=13$ see \eref{PPG13PT}}.
\end{example}

\section{Oriented Riordan graph of the Bell type}

\begin{definition}\label{faithful-def} {\rm Let $G_n^\sigma=G_n^\sigma(g,f)$ be a proper oriented Riordan graph with the vertex sets $V_i=\{j\in[n]\;|\;j\equiv i\}$, $i=1,2,3$. If $\left<
V_{1}\right> \cong G_{\lceil n/3\rceil}^\sigma$ and $\left<
V_{2}\right>$ and $\left< V_{3}\right>$ are null graphs then
$G_n^\sigma$ is {\em i1-decomposable}.}
\end{definition}

An oriented Riordan graph $G_n^\sigma(g,f)$ is called of {\it Bell
type} if $f=zg$.

\begin{lemma}\label{lem1}
Let $G_n^\sigma(g,zg)$ be a proper. Then the induced subgraph
$\left<V_3\right>$ is a null graph.
\end{lemma}
\begin{proof}
Since $[z^{3m+1}]g(z)(zg(z))^{p-1}=[z^{3m-1}]g^p(z)\equiv
[z^{3m-1}]g(z^3)\equiv0$, it follows from (ii) of Theorem
\ref{e:th13} that the induced subgraph $\left\langle
V_3\right\rangle$ is a null graph. Hence the proof
follows.\end{proof}

\begin{theorem}\label{i1-decomposable-Theorem}
An oriented Riordan graph $G_n^{\sigma}(g,zg)$ is i1-decomposable if
and only if
\begin{align}\label{eq5}
g'\equiv\pm g^2.
\end{align}
\end{theorem}
\begin{proof} Let $V_i=\{j\in[n]\;|\;j\equiv i\}$ be the vertex
subsets of $G_n^{\sigma}(g,zg)$ for $i=1,2,3$. From Lemma
\ref{lem1}, $\left<V_3\right>$ is the null graph. By definition,
$G_n^{\sigma}(g,zg)$ is i1-decomposable if and only if $\left<
V_{1}\right> \cong G_{\lceil n/3\rceil}^\sigma(g,zg)$ and $\left<
V_{2}\right>$ is the null graph. By Oriented Riordan Graph
Decomposition, $\left< V_{1}\right> \cong
G_{\lceil(n-1)/3\rceil}^\sigma(g,zg)$ and $\left< V_{2}\right>$ is
the null graph if and only if
\begin{align*}
&-g''(\sqrt[3]{z})\equiv g(z)\;\;{\rm
and}\;\;-(g^2)''(\sqrt[3]{z})\equiv 0\\
\Leftrightarrow\;\;&g''(z)\equiv -g(z^3)\equiv  -g^3(z)\;\;{\rm
and}\;\;(g')^2\equiv -g''g\equiv g^4\\
\Leftrightarrow\;\;&g'\equiv \pm g^2.
\end{align*}
Hence the proof follows.
\end{proof}

\begin{theorem}\label{i1-decomposable-A-seq}
An oriented Riordan graph $G_n^{\sigma}(g,zg)$ is i1-decomposable if
and only if the ternary $A$-sequence $A=(a_k)_{k\ge0}$ of
$G_n^{\sigma}(g,zg)$ satisfies either
\begin{align}\label{A-seq-i1}
(1,1,0,a_3,a_3,0,a_6,a_6,0,\ldots),\quad a_i\in\{0,1,-1\}.
\end{align}
or
\begin{align}\label{A-seq-i1-1}
(1,-1,0,a_3,-a_3,0,a_6,-a_6,0,\ldots),\quad a_i\in\{0,1,-1\}.
\end{align}
\end{theorem}
\begin{proof}
Let $G_n^{\sigma}(g,zg)$ be i1-decomposable. Since there is a unique
generating function $A=\sum_{i\ge0}a_iz^i$ such that $g=A(zg)$, by
applying derivative to both sides, we obtain
\begin{align}\label{eq1}
g'\equiv(g+zg')\cdot A'(zg).
\end{align}
By Theorem \ref{i1-decomposable-Theorem}, the equation \eref{eq1} is
equivalent to
\begin{align}\label{eq2}
 g\equiv(\pm1+ zg)\cdot A'(zg),\;\;{\rm i.e.}\;
A(zg)\equiv(\pm1+ zg)\cdot A'(zg)
\end{align}
Let $f=zg$ and $A(z)=\sum_{n\ge0}a_nz^n$ with $a_0=1$. Since
$[z^0]f=0$ and $[z^1]f=1$, there is a composition inverse of $f$ so
that the equation \eref{eq2} is equivalent to either
\begin{align*}
\sum_{n\ge0}a_nz^n\equiv(1+z)\cdot
A'(z)\equiv\sum_{n\ge0}\left(a_{3i+1}+(a_{3i+1}-a_{3i+2})z-a_{3i+2}z^2\right)z^{3i}
\end{align*}
or
\begin{align*}
\sum_{n\ge0}a_nz^n\equiv(-1+ z)\cdot
A'(z)\equiv\sum_{n\ge0}\left(-a_{3i+1}+(a_{3i+1}+a_{3i+2})z-a_{3i+2}z^2\right)z^{3i}
\end{align*}
which implies the desired result.
\end{proof}

\begin{example}\label{ex}
{\rm Let $PG_n^\sigma=G_n^{\sigma}({1\over 1-z},{z\over 1-z})$ and
$CG_n^\sigma=G_n^{\sigma}(C,zC)$ where $C={1-\sqrt{1-4z}\over 2z}$
is the Catalan generating function. Since $A$-sequences of
$PG_n^\sigma$ and $CG_n^\sigma$ are $(1,1,0,\ldots)$ and
$(1,1,1,\ldots)$ respectively, by Theorem
\ref{i1-decomposable-A-seq} $PG_n^\sigma$ is i1-decomposable but
$CG_n^\sigma$ is not.}
\end{example}

\begin{theorem}\label{th1}
Let $G_n^\sigma=G_n^{\sigma}(g,zg)$ be an i1-decomposable graph with
$g'\equiv g^2$. Then we have the following:
\begin{itemize}
\item[(i)] If $n=3^{i}+1$ for $i\ge1$ then $n$th low of skew-adjacency matrix $\mathcal{S}(G_n^\sigma)$ is
given by
\begin{align*}
(1,-1,\ldots,1,-1,1,0);
\end{align*}
\item[(ii)] If $n=2\cdot3^{i}+1$ for $i\ge0$ then $n$th low of skew-adjacency matrix $\mathcal{S}(G_n^\sigma)$ is
given by
\begin{align*}
(\underbrace{1,-1,\ldots,1,-1,1}_{3^i\;{\rm
term}},\underbrace{1,-1,\ldots,1,-1,1}_{3^i\;{\rm term}},0).
\end{align*}
\end{itemize}
\end{theorem}
\begin{proof} (i) Let $\mathcal{S}(G_n^\sigma)=(s_{i,j})_{1\le i,j\le n}$ and $n=3^i+1$. First we show that
$s_{n,k}=-s_{n,k}$ for $1\le k<n-1$. There exits an integer
$t\not\equiv0$ such that $k=3^st$ for some nonnegative integer $s<
i$. Since $[z^{m-1}]g'=m[z^m]g$ and $g'\equiv g^2$, we obtain
\begin{align*}
ts_{n,k}&\equiv t[z^{3^i-1}]z^{k-1}g^k\equiv -{(3^i-k)[z^{3^i-k}]g^k\over 3^s}=-{[z^{3^i-k-1}](g^k)'\over 3^s}\\
&=-{k[z^{3^i-k-1}]g^{k-1}g'\over
3^s}\equiv-t[z^{3^i-1}]z^{k}g^{k+1}\equiv-ts_{n,k+1}.
\end{align*}
Thus we have $s_{n,k}=-s_{n,k+1}$ for all $k=1,\ldots,n-1$.

Now it is enough to show that $s_{n,1}\equiv[z^{3^i-1}]g\equiv1$.
Since
\begin{align*}
g'(z)\equiv
g^2(z)\;\Rightarrow\;g''(z)=2g(z)g'(z)\equiv-g^3(z)\equiv-g(z^3),
\end{align*}
by comparing the coefficients of $g''(z)$ and $g(z^3)$ we obtain
\begin{align*}
[z^j]g\equiv [z^{3j+2}]g\;\Rightarrow\;[z^{3^j-1}]g\equiv[z^0]g=1.
\end{align*}
Hence we complete the proof.

 (ii) Let $\mathcal{S}(G_n^\sigma)=(s_{i,j})_{1\le i,j\le n}$ and $n=2\cdot3^i+1$. First we show that
$s_{n,k}=-s_{n,k}$ for $1\le k\le3^i$. There exits an integer
$t\not\equiv0$ such that $k=3^st$ for some nonnegative ineger $s<
i$. Since $[z^{m-1}]g'=m[z^m]g$ and $g'\equiv g^2$, we obtain
\begin{align*}
ts_{n,k}&\equiv t[z^{2\cdot3^i-1}]z^{k-1}g^k\equiv -{(2\cdot3^i-k)[z^{2\cdot3^i-k}]g^k\over 3^s}=-{[z^{2\cdot3^i-k-1}](g^k)'\over 3^s}\\
&=-{k[z^{2\cdot3^i-k-1}]g^{k-1}g'\over
3^s}\equiv-t[z^{2\cdot3^i-1}]z^{k}g^{k+1}\equiv-ts_{n,k+1}.
\end{align*}
Thus we obtain
\begin{align}\label{eq6}
s_{n,k}=-s_{n,k}\;\;(1\le k<3^i).
\end{align}
Similarly, we can show that
\begin{align}\label{eq7}
s_{n,k}=-s_{n,k}\;\;(3^i+1\le k\le 2\cdot3^i).
\end{align}Since
\begin{align*}
s_{n,3^i}&\equiv [z^{2\cdot3^i-1}]z^{3^i-1}g^{3^i}\equiv{3^i[z^{3^i}]g^{3^i}\over 3^i}={[z^{3^i-1}](g^{3^i})'\over 3^i}\\
&={3^i[z^{3^i-1}]g^{3^i-1}g'\over
3^i}\equiv[z^{2\cdot3^i-1}]z^{3^i}g^{3^i+1}\equiv s_{n,3^i+1},
\end{align*}
by \eref{eq6} and \eref{eq7} $n$th low of skew-adjacency matrix
$\mathcal{S}(G_n^\sigma)$ is given by
\begin{align*}
(\underbrace{x,-x,\ldots,x,-x,x}_{3^i\;{\rm
term}},\underbrace{x,-x,\ldots,x,-x,x}_{3^i\;{\rm term}},0)
\end{align*}
where $x=s_{n,1}$.

Now it is enough to show that
$s_{n,1}\equiv[z^{2\cdot3^i-1}]g\equiv1$. Since from \eref{A-seq-i1}
we have $[g^1]g=[z^0]g=1$, by comparing the coefficients of $g''(z)$
and $g(z^3)$ we obtain $[z^{2\cdot3^j-1}]g\equiv[z^1]g=1$. Hence we
complete the proof.
\end{proof}

\begin{example}
{\rm Let us consider the oriented Pascal graph
$PG_n^\sigma=G_n^{\sigma}({1\over 1-z},{z\over 1-z})$. Since
$PG_n^\sigma$ is i1-decomposable by Example \ref{ex} and $g'={1\over
1-z}'={1\over (1-z)^2}=g^2$, it follows from (ii) of Theorem
\ref{th1} that the 19th low of $PG_{19}$ is given by
\begin{align*}
(1,-1,1,-1,-1,1,-1,-1,1,1,-1,1,-1,-1,1,-1,-1,1,0),\;\; \textrm{see
Figure \ref{PG19}}.
\end{align*}
}
\end{example}

By using the similar argument of Theorem \ref{th1}, we obtain the
following result.

\begin{theorem}\label{th2}
Let $G_n^\sigma=G_n^\sigma(g,zg)$ be an i1-decomposable graph with
$g'\equiv-g^2$. Then we have the following:
\begin{itemize}
\item[(i)] If $n=3^{i}+1$ for $i\ge1$ then $n$th low of skew-adjacency matrix $\mathcal{S}(G_n^\sigma)$ is
given by
\begin{align*}
(1,1,\ldots,1,0);
\end{align*}
\item[(ii)] If $n=2\cdot3^{i}+1$ for $i\ge0$ then $n$th low of skew-adjacency matrix $\mathcal{S}(G_n^\sigma)$ is
given by
\begin{align*}
(\underbrace{-1,-1,\ldots,-1}_{3^i\;{\rm
term}},\underbrace{1,1,\ldots,1}_{3^i\;{\rm term}},0).
\end{align*}
\end{itemize}
\end{theorem}
\begin{proof}
 (i) Let $\mathcal{S}(G_n^\sigma)=(s_{i,j})_{1\le i,j\le n}$ and $n=3^i+1$. First we show that
$s_{n,k}=s_{n,k}$ for $1\le k<n-1$. There exits an integer
$t\not\equiv0$ such that $k=3^st$ for some nonnegative integer $s<
i$. Since $[z^{m-1}]g'=m[z^m]g$ and $g'\equiv-g^2$, we obtain
\begin{align*}
ts_{n,k}&\equiv t[z^{3^i-1}]z^{k-1}g^k\equiv -{(3^i-k)[z^{3^i-k}]g^k\over 3^s}=-{[z^{3^i-k-1}](g^k)'\over 3^s}\\
&=-{k[z^{3^i-k-1}]g^{k-1}g'\over 3^s}\equiv
t[z^{3^i-1}]z^{k}g^{k+1}\equiv ts_{n,k+1}.
\end{align*}
Thus we have $s_{n,k}=s_{n,k+1}$ for all $k=1,\ldots,n-1$.

Now it is enough to show that $s_{n,1}\equiv[z^{3^i-1}]g\equiv1$.
Since
\begin{align*}
g'(z)\equiv -g^2(z)\;\Rightarrow\;g''(z)=-2g(z)g'(z)\equiv
-g^3(z)\equiv -g(z^3),
\end{align*}
by comparing the coefficients of $g''(z)$ and $g(z^3)$ we obtain
\begin{align*}
[z^{3j+2}]g\equiv [z^j]g,\;j\ge0
\;\Rightarrow\;[z^{3^i-1}]g\equiv[z^0]g=1,\;i\ge0.
\end{align*}
Hence we complete the proof.
\end{proof}

\section{$p$-Riordan graphs}

\begin{definition}\cite{GX} {\rm Let $G^\sigma$ be a simple weighted undirected graph with an orientation
$\sigma$, which assigns to each edge a direction so that $G^\sigma$
becomes a {\em weighted oriented graph}. (We will refer to an
unweighted oriented graph which is just a weighted oriented graph
with weight of each arc equals to 1.) The skew-adjacency matrix
associated to the weighted oriented graph $G^\sigma$ with the vertex
set $[n]$ is defined as the $n\times n$ matrix
$\mathcal{S}(G^\sigma)=(s_{i,j})_{1\le i,j\le n}$ whose
$(i,j)$-entry satisfies:
\begin{align*}
s_{i,j}=\left\{
\begin{array}{lll}
\omega,&
\text{if there is an arc with weight $\omega$ from $i$ to $j$;} \\
-\omega, & \text{if there is an arc with weight $\omega$ from $j$ to
$i$;} \\
0, & \text{otherwise.}
\end{array}
\right.
\end{align*}
In particular, when $\omega=1$ the weighted oriented graph
$G^\sigma$ is the oriented graph.}
\end{definition}

\begin{definition}
{\rm Let $p\ge3$ be prime.  An oriented graph $G^\sigma$ with $n$
vertices is called a \emph{$p$-Riordan graph} if there exists a labeling $1,2,\ldots,n$ of
$G^\sigma$ such that its skew-adjacency matrix
$\mathcal{S}(G^\sigma)=[s_{ij}]_{1\le i,j\le n}$ is given by
\begin{align*}
\mathcal{S}(G^\sigma)\equiv (zg,f)_n-(zg,f)_n^T\;({\rm mod p})\;\;
{\rm and}\;\;s_{i,j}\in\{\lfloor
p/2\rfloor,\ldots,-1,0,1,\ldots,\lfloor p/2\rfloor\}.
\end{align*}
We denote such graph by $G^\sigma_{n,p}(g,f)$, or simply by
$G^\sigma_{n,p}$ when the Riordan array $(g,f)$ is understood from
the context, or it is not important.}
\end{definition}

\begin{proposition}
The number of weighted oriented $p$-Riordan graphs of order $n\ge1$
is
\begin{align*}
{p^{2(n-1)}+p\over p+1}.
\end{align*}
\end{proposition}

\begin{definition}{\rm Let $G^\sigma_{n,p}=[s_{i,j}]_{1\le i,j\le n}$ be a weighted oriented $p$-Riordan graph. A pair of vertices $\{k,t\}$ in $G$ is
a {\em weighted cognate pair} with a pair of vertices $\{i,j\}$ in
$G$ if
\begin{itemize}
\item $|i-j|=|k-t|$ and
\item an arc $i\rightarrow j$ with the weight $s_{i,j}$ if and only if an arc $k\rightarrow t$ with the weight $s_{i,j}$.
\end{itemize}
The set of all weighted cognate pairs of $\{i,j\}$ is denoted by
wcog$(i,j)$.}
\end{definition}

In this section, we simply denote $a\equiv_p b$ if $a\equiv_p b$
(mod $p$).

\begin{lemma}\label{p-power}
Let $g,f\in\mathbb{Z}[[z]]$ with $f(0)=0$. For a prime $p$ and
$k\ge0$, we obtain
\begin{eqnarray*}
g^{p^k}(f)\equiv_p g(f^{p^k}).
\end{eqnarray*}
\end{lemma}

The following theorem gives a relationship between weighted cognate
pair cognate pairs and the $A$-sequence of a Riordan graph.

\begin{theorem}\label{e:cognate} For $\ell\ge 0$, let
$A=(a_k)_{k\ge0}=(a_0,\underbrace{0,\ldots,0}_{\ell \;{\rm
times}},a_{\ell+1},a_{\ell+2},\ldots)$ with
$a_0,a_{\ell+1}\not\equiv_p0$ be the $p$-ary $A$-sequence for a
weighted oriented $p$-Riordan graph $G_{n,p}^\sigma(g,f)$ where
$f\ne z$ and $f'(0)=1$. Then
\begin{align*}
\mbox{\em wcog}(i,j)=\left\{\left\{i+mp^{s}, j+mp^{s}\right\}\ |\
i+mp^{s},\ j+mp^{s}\in[n]\right\}
\end{align*}
 where $s\ge0$ is an integer with $p\lfloor(|i-j|-1)/p^s\rfloor\le \ell$.
\end{theorem}

\begin{theorem}\label{e:coro2}
For $\ell\ge 0$, let
$A=(a_k)_{k\ge0}=(1,\underbrace{0,\ldots,0}_{\ell \;{\rm
times}},a_{\ell+1},a_{\ell+2},\ldots)$ with
$a_{\ell+1}\not\equiv_p0$ be the $p$-ary $A$-sequence for a weighted
oriented Riordan graph $G_{n,p}^\sigma=G_{n,p}^\sigma(g,f)$ where
$f\ne z$ and $f'(0)=1$. For each $s\ge 0$ and
$k\in\{0,\ldots,\ell\}$, $G_{n,p}^\sigma$ has the following fractal
properties:
\begin{itemize}
\item[{\rm(i)}] $\left<\{1,\ldots,(k+1)p^{s}+1\}\right>\cong
\left<\{\alpha(k+1)p^{s}+1,\ldots,(\alpha+1)(k+1)p^{s}+1\}\right>$
\item[{\rm(ii)}]$\left<\{1,\ldots,(k+1)p^{s}\}\right>\cong
\left<\{\alpha(k+1)p^{s}+1,\ldots,(\alpha+1)(k+1)p^{s}\}\right>$
\end{itemize}
where $\alpha\ge1$.
\end{theorem}

\begin{theorem}  Let $G_{n,p}^\sigma=G_{n,p}^\sigma(g,f)$ be a weighted oriented $p$-Riordan graph of order $n\ge p$ and $V_i=\{j\in[n]\;|\;j\equiv_pi\}$.
If $G_n^{(p)}$ is proper then
\begin{itemize}
\item[{\rm(i)}] There exists a permutation matrix $P$ such that the skew-adjacency matrix $\mathcal{S}(G_n^{(p)})$ satisfies
\begin{eqnarray}\label{e:bm-p}
\mathcal{S}(G_{n,p}^\sigma)=P^{T}\left(
\begin{array}{cccc}
X_1 & -B_{1,2}  &\cdots & -B_{1,p}  \\
B_{1,2}^T & X_2 & \ddots &\vdots \\
\vdots & \ddots &\ddots &-B_{p-1,p}\\
B_{1,p}^T &\cdots & B_{p-1,p}^T & X_p
\end{array}
\right)P
\end{eqnarray}
where $P=\left[e_{1}|e_{p+1}|\cdots|e_{p\lceil n/p\rceil
-p+1}|\cdots|e_{p}|e_{2p} \;|\; \cdots \;|\; e_{p\lfloor n/p\rfloor
}\right] ^{T}$ is the $n\times n$ permutation matrix and $e_{i}$ is
the elementary column vector with the $i$th entry being $1$ and the
others entries being $0$.
\item[{\rm(ii)}] The matrix $X_t$ is the skew-adjacency matrix of the
induced subgraph of $G_{n,p}^\sigma$ by $V_t$. In particular, the
induced subgraph $\left<V_t\right>$ is isomorphic to a weighted
oriented $p$-Riordan graph of order $\ell_t$ given~by
\begin{align*}
G_{\ell_t,p}^\sigma\left(-{d^{p-1}\over
dz^{p-1}}\left({gf^{t-1}\over
z^{t-1}}\right)(\sqrt[p]{z}),f(z)\right).
\end{align*}
\item[{\rm(iii)}] The matrix $B_{t,k}$ representing the edges between $V_t$ and $V_k$
can be expressed as the sum of two matrices as follows:
\begin{align*}
B_{t,k}\equiv_p\left(-z{d^{p-1}\over dz^{p-1}}\left({gf^{k-1}\over
z^{t-1}}\right)(\sqrt[p]{z}),f(z)\right)_{\ell_t
\times\ell_k}+\left({d^{p-1}\over
dz^{p-1}}(z^{p-k+1}gf^{t-1})(\sqrt{z}),f(z)\right)_{\ell_k\times\ell_t}^{T}.
\end{align*}
\end{itemize}
where $\ell_t=\lfloor (n-t)/p\rfloor+1$.
\end{theorem}

\end{document}